%% file: mr-arxiv.tex
\documentclass{article}
\usepackage[margin=4cm]{geometry}
\RequirePackage[colorlinks,citecolor=blue,urlcolor=blue]{hyperref}

\usepackage[utf8]{inputenc} 
\usepackage[T1]{fontenc}    
\usepackage{hyperref}       
\usepackage{url}            
\usepackage{booktabs}       
\usepackage{amsfonts}       
\usepackage{nicefrac}       
\usepackage{microtype}      
\usepackage{enumitem}

\title{On the support recovery of marginal regression}
\usepackage{times}

\usepackage{natbib}
\usepackage{amssymb}
\usepackage{amsmath}
\usepackage{amsthm}
\usepackage{verbatim}
\usepackage{graphicx}
\usepackage{epstopdf}
\usepackage{amsmath}
\usepackage{bm}
\usepackage{setspace}

\usepackage{algorithm}
\usepackage{algpseudocode}

\algnewcommand{\IIf}[1]{\State\algorithmicif\ #1\ \algorithmicthen}

\usepackage[matrixnorms,opnorm=2]{arash_macros}

\newcommand{\R}{\mathbb{R}}

\input{defs}

\newcommand\keywords[1]{\textbf{Keywords}:#1}
\author{S. Jalil Kazemitabar, Arash A. Amini and
	Ameet Talwalkar}
\begin{document}

\maketitle

	\begin{abstract}
		Leading methods for support recovery in high-dimensional regression, such as Lasso, have been well-studied and their limitations in the context of correlated design have been characterized with precise incoherence conditions. In this work, we present a similar treatment of selection consistency for marginal regression (MR), a computationally efficient family of methods with connections to decision trees. Selection based on marginal regression is also referred to as covariate screening or independence screening and is a popular approach in applied work, especially in ultra high-dimensional settings.
		  We identify the underlying factors---which we denote as \emph{MR incoherence}---affecting MR's support recovery performance. Our near complete characterization provides a much more nuanced and optimistic view of MR in comparison to previous works. 
To ground our results, we provide a broad taxonomy of results for leading feature selection methods, relating the behavior of Lasso, OMP, SIS, and MR.  We also lay the foundation for interesting generalizations of our analysis, e.g., to non-linear feature selection methods and to more general regression frameworks such as a general additive models.

	\medskip
	\keywords{	Support recovery; high-dimensional regression; marginal regression; independence screeing; covariate screeing; incoherence condition}
	\end{abstract}

\input{intro_ameet}

\input{related_work}

\input{amendment}

\paragraph{Notation.} For any vector $\beta \in \reals^p$, let us write $ \minnorm{\beta} = \min_{j \in [p]} |\beta_j|$. For $S \subset [p]$, $\beta_S = (\beta_i, i\in S)$ is the subvector of $\beta$ on indices $S$, and hence $\minnorm{\beta_S} = \min_{j \in S} |\beta_j|$. Similar notation is used for sub-blocks of matrices, e.g., $\Sigma_{SS^c}$ is the block of $\Sigma$ indexed by rows $S$ and columns $S^c$. We write $\mnorm{A}_p$ for the $\ell_p$ operator norm of a matrix. For example, $\mnorm{A}_\infty$ is the  $\ell_\infty$ operator norm which is equal to the maximum absolute row sum. Similarly, $\opnorm{A}$ is the usual $\ell_2$ operator norm. We use $A_{i*}$ and $A_{*j}$ to denote the $i$th row and the $j$th column of $A$, respectively. The symbols $\lesssim$ and $\gtrsim$ are used to denote inequalities up to constants. 

\section{Selection by marginal regression}
 In this section, we first formalize the exact support recovery problem and state mutual incoherence conditions for MR to recover the support uniformly over a controlled class of parameters. We study the problem first at the population level, and then extend the results to the finite sample regime. 

\input{mr_population}

\input{mr_sampling}
\input{landscape}


\input{discussion}
\input{extension}
\bibliography{refs}

\appendix

\input{lasso_incoh}

\input{appendices}

\end{document}

%% file: defs.tex
\newcommand\bett{\widetilde{\beta}}
\newcommand\bets{\beta^*}
\newcommand\Bc{\mathcal B}

\DeclareMathOperator{\cov}{cov}

\newcommand\alg{\mathcal A}
\newcommand\Pc{\mathcal P}
\newcommand\Pclin{\Pc_{\text{lin}}}
\DeclareMathOperator{\mri}{MRI}
\DeclareMathOperator{\lai}{LAI}
\newcommand\minnorm[1]{|#1|_{\min}}

\newcommand\opinfnorm[1]{\mnorm{#1}_{\infty}}

\newcommand{\x}{x}

\newcommand{\pcor}{r}
\newcommand{\scor}{\widehat{r}}
\newcommand\dpw{\delta_\text{PW}}
\newcommand\St{\widetilde{S}}

%% file: intro_ameet.tex
\section{Introduction} \label{sec:intro}
Support recovery in high-dimensional regression is a well-studied problem, and of significant practical importance, e.g., in the context of model interpretability.  Leading methods such as Lasso and other $\ell_1$-regularization  variants~\citep{hastie2015statistical} have computational complexity $O(p^2n)$ which introduces significant computational overhead for large $p$. These methods  demonstrate limitations for support recovery when correlation in the design exceeds moderate amounts, as characterized by \emph{Lasso incoherence}. 

We focus on an alternative approach, \emph{marginal regression} (MR), which is attractive for its algorithmic simplicity, computational efficiency, and its capability for embarrassing parallelism.  In particular, MR independently compares each covariate to the response, following a procedure that closely resembles the splitting criterion used for decision trees~\citep{Kazemitabar2017}. The greedy nature of this approach suggests that it may be subject to significantly more onerous limitations on correlation in the design. It is worth noting that selection based on marginal regression goes by many other names such covariate screening or independence screening, and is a popular approach in applied work, especially in ultra high-dimensional settings.

Our primary contribution in this paper is showing that the conditions on MR support recovery are not nearly as pessimistic as have been previously assumed. In particular, while \emph{pairwise incoherence (PWI)} is known to be a sufficient condition~\citep{donoho2001uncertainty,Donoho2003}, we show that it is overly stringent. 
Our results demonstrate that the behavior of MR is much more nuanced than  what PWI predicts. We establish this claim by providing a near complete characterization of the support recovery performance of MR, revealing the role of various parameters and drawing some surprising conclusions about the strength of MR in certain situations (and its weaknesses in others). 

More precisely, we derive a condition on the covariance matrix $\Sigma \in \reals^{p \times p}$ of the variables which we call \emph{MR incoherence}. We also introduce a parameter $R \in [1,\infty]$ that controls the spread of the regression coefficients $\beta$, that is: $|\beta_i | \le R |\beta_j|$ for all $i,j \in [p]$. We show that MR incoherence is necessary and sufficient for recovery when $R \ge 2$ and sufficient when $R \in [1,2)$. Our results also hinge on the correlation among the on-support variables, namely the $\Sigma_{SS}$ block of the covariance matrix $\Sigma$---where $S$ is the support of $\beta$.
There are several remarkable consequences of our analysis:
\begin{itemize}[leftmargin=*]
\item MR can benefit from sparsity of the covariance structure, i.e., if the on-support covariates correlate in pairs (together with a low spread of the coefficient parameters), MR has a comparable performance to that of the Lasso. More generally, when on-support covariates correlate in groups of size $r$, MR shows a substantial deviation in performance from PWI and compete closer to that of Lasso.
\item When the on-support covariates are nearly uncorrelated and the regression coefficients are not spread out, the support recovery performance of MR again approaches that of the Lasso; this fact is independent of the correlation between off-support and on-support variables.
\item MR incoherence implies \emph{restricted isometry property} (RIP), a known sufficient condition for Lasso. 
\item The uniform performance of MR is sensitive to the minimum eigenvalue of the covariance matrix, which contradicts the intuition from prior work that does not consider   uniform recovery within a class of parameters~\citep{Genovese2012}.
\end{itemize}

We owe these results to our novel approach which incorporates the spread of the regression coefficients. Prior work either neglected it as a factor~\citep{Genovese2012} or failed to separate its effect from the covariance matrix~\citep{Fan2008}, which led to rejecting the possibility of uniform recovery for MR~\citep{robins2003uniform}. Our novel approach lays the foundation for interesting generalizations of our analysis.  In Section~\ref{sec:landscape}, we leverage our results to present a broad taxonomy of the necessary and sufficient conditions for a wide range of feature selection methods, and relate the performance of marginal regression with popular feature selection methods including Lasso, (Orthogonal) Matching Pursuit, and SIS.  Moreover, as we discuss in Section~\ref{sec:discuss}, it is possible to extend our results to non-linear feature selection methods such as Dstump~\citep{Kazemitabar2017} or information gain, and to more general regression frameworks such as a general additive model under suitable regularity conditions. 

%% file: related_work.tex
\paragraph{Related work.} 
In this work we study uniform recovery by MR, as opposed to previous works which studied the average case and ``a fixed single parameter'' cases. 
We precisely characterize how the spread of regression coefficients and the structure of the on-support covariance plays a role in the selection consistency of MR. The necessary and sufficient conditions we provide are new to the best of our knowledge and are more relaxed as well as clearer than earlier results.

 That a form of PWI is sufficient for MR support recovery, in the fixed design regression setting, is colloquially known and often attributed to the work by~\cite{donoho2001uncertainty} and~\cite{Donoho2003}, although we could not find this exact result there or in the literature. 
  Theorem~\ref{thm:MR:main:res}, which gives the necessary and sufficient conditions for recovery in terms of  MR incoherence (cf.~Definition~\ref{def:MR:inc}) clearly
   shows that conditions needed for MR recovery are in general  weaker than PWI. 
   Moreover, if $R = \infty$, no amount of PWI can be tolerated by MR, countering the colloquial knowledge.  
 

\cite{Genovese2012} showed that uniform recovery by MR is not possible when $R$ is unbounded (except in the trivial case of $\Sigma = I$). This is the content of Lemma~\ref{lem:mr:incoh:R:infinity} in our work. Due to this negative result, the bulk of the work in~\cite{Genovese2012} focused on average case recovery (i.e., putting a sparse prior on $\beta$, the regression coefficient vector, and recovering the support with high probability).  In contrast, we show that it is possible to recover uniformly over a class of coefficients, assuming $R < \infty$.


\cite{Fan2008} is among the earliest and most noted work on MR. They termed the approach \emph{sure independence screening (SIS)} and provided sufficient asymptotic guarantees, as the sample size $n\to \infty$, for SIS to recover a superset guaranteed to contain the true support with high probability. The sufficient conditions in~\cite{Fan2008} are tangled with other assumptions on the sampling process and hence hard to compare with known incoherence-based results.  

Finally, we note that by using Pearson correlation for importance scoring, MR  can be viewed as a filter method, i.e., a method that independently scores covariates based on their relevance to the target. In contrast, a wide range of (forward) wrapper feature selection methods iteratively evaluate the importance of covariates by including them in the model one after another, with each iterative decision typically based on some importance score. Matching and orthogonal matching pursuit~\citep{donoho2001uncertainty,Donoho2003} are examples of wrappers that use Pearson correlation, just as in MR, but adjust for the interdependence of the covariates by greedily selecting the most important ones and working with the residuals. Lasso can also be implemented in a similar way, often referred to as forward stagewise regression~\citep{hastie2007forward}. 
In Section~\ref{sec:landscape} we leverage our novel results to connect MR with these wrapper methods, and more broadly to provide a taxonomy of existing and conjectured results, as well as open questions.  As part of this discussion, we also relate MR to Sure Independence Screening~\citep{Fan2008}, a method closely related to MR but  solving an easier problem.

%% file: amendment.tex
\paragraph{Amendment to related work.} 
After writing this paper, we noticed that an earlier work~\citep{wang2015consistency} provides results comparable to what we have achieved. The authors in that paper define an incoherence condition, named restricted diagonally dominant (RDD) condition, that is close to our main condition (Definition~\ref{def:MR:inc}) in form. They show that RDD is a tight guarantee for uniform recovery using marginal regression. Our work  deviates from theirs in that they consider \emph{signed support recovery (SSR)} whereas we consider just the \emph{support recovery (SR)}, i.e., we do not require the sign of coefficients to be recovered correctly. In this sense, our work is complementary to~\cite{wang2015consistency}. 
 Their condition (RDD) is strictly stronger than MRI (that is, RDD $\implies$ MRI) and the extra strength maps to the ``sign recovery'' part of the problem.

Sign recovery is often assumed as a technical device since it greatly simplify the analysis. Here, we directly derive necessary and sufficient conditions for SR without any sign requirement leading to the MRI conditions. Considering the work of~\cite{wang2015consistency} and ours together reveals an interesting point: Dropping the sign requirement changes the nature of the problem; while RDD is necessary and sufficient for SSR for the whole range of $R > 0$ (see~\eqref{eq:Gamma:S:def} for the definition of $R$), MRI is only so for $R \in [2,\infty)$; for $R \in (1,2)$ the necessary and sufficient condition for SR will be combinatorial. That is, there is a dichotomy in the SR problem which is not in SSR. This requires  non-elementary arguments in our case as opposed to the short analysis of~\cite{wang2015consistency} (e.g., the sufficiency proof in~\cite{wang2015consistency} does not go through if RDD is replaced with MRI.) We believe our technical contributions here will be of interest since as far as we know, no other necessary and sufficient condition of the SR type is available (even for the Lasso).

%% file: mr_population.tex
\subsection{Marginal regression at the population level}
\label{sec:MR:population}

At the population level, a  random design linear model with response $Y \in \reals$,  covariate (or feature) vector $X = (X_1,\dots,X_p) \in \reals^p$ and noise  $\eps \in \reals$, is of the form
\begin{align}\label{eq:lin:model:pop}
	\begin{split}
		Y = \beta^T X + \eps, \quad 
		\text{where}
		\quad \ex(\eps) &= 0,\quad \var(\eps) = \sigma^2 ,\\
		\quad \ex(X) &= 0_p,  \quad \cov(X) = \Sigma, \quad \Sigma_{ii} = 1, \forall i \in [p], \\
		\quad \cov(X,\eps) &= 0_p
	\end{split}
\end{align}
and in addition we assume $\ex(X) = 0$ and $\cov(X,\eps) = 0$, i.e., the noise and the covariate vector are uncorrelated. Note that we are working with a random design model, i.e., $X$ is a random variable. The covariance matrix of $X$, namely, $\Sigma$ will play the prominent role in the support recovery conditions presented here. The diagonal scaling $\Sigma_{ii} = 1$ is natural for studying a correlation based approach such as marginal regression and it is inline with common practice of standardizing covariates before performing regression.


Let us fix a subset $S \subset [p]$ with $|S|= s$, which will serve as the true support of $\beta$ to be recovered. We will assume that $s$ is known and available to the algorithms. The class of parameters of interest  in this paper is:
\begin{align}\label{eq:Gamma:S:def}
\Gamma_S := \Gamma_{S,\rho,R} := \big\{\beta \in \R^p \;|\; \beta_{S^c}= 0, \;\;
\minnorm{\beta_S} > \rho, \;\; \|\beta_S\|_\infty \leq R \minnorm{\beta_S}\big\} 
\end{align}
where $R \in [1,\infty]$.
Note that the support of any $\beta \in \Gamma_S$ is contained in $S$. The parameters $\rho$ and $R$ control, respectively, the so-called minimum signal strength and the spread of the on-support parameters. The two extremes $R=1$ versus $R=\infty$ correspond to equal magnitude for the on-support elements versus no restriction on the relative sizes of $\beta_j,j\in S$. In other words, for $R=1$, $\beta_S$ is a scaled multiple of a sign vector: $\beta_S \in \bigcup_{ t \ge \rho} \{-t,t\}^s $.

The (worst-case) support recovery problem over $\Gamma_S$ can be stated as follows: Given that $(Y,X)$ follows model~\eqref{eq:lin:model:pop} with $\beta \in \Gamma_S$, can we recover the support of $\beta$, i.e., the set of indices of its nonzero elements?
The guarantee of recovery should hold uniformly over $\beta \in \Gamma_S$. To make this notion more precise, let  $\Pclin^{p+1}$ be the class of distributions for $(Y,X)$ that satisfy~\eqref{eq:lin:model:pop}. We write $\pr_{\beta,\Sigma} \in \Pclin^{p+1}$ for any distribution for $(Y,X)$ that satisfies~\eqref{eq:lin:model:pop} with regression coefficient $\beta$ and feature covariance $\Sigma$.

A population level support recovery algorithm $\alg$ takes a distribution $P_{\beta,\Sigma}$ and outputs a subset of $[p]$ that is believed to be the support of $\beta$. Formally, such an algorithm is a map $\alg : \Pclin^{p+1} \to [p]$. We say that $\alg$ succeeds in support recovery uniformly over $\Gamma_S$ if
\begin{align}\label{eq:pop:supp:recovery:def}
	\alg(\pr_{\beta,\Sigma}) = \supp(\beta), \quad \forall \pr_{\beta,\Sigma} \in \Pclin^{p+1},\; \beta \in \Gamma_S.
\end{align}
If~\eqref{eq:pop:supp:recovery:def} holds, we also say that $\alg$ is model selection consistent over $\Gamma_S$ at the population level.
Ideally one would like~\eqref{eq:pop:supp:recovery:def} to hold for all nonsingular $\Sigma$ as well. However, for any particular algorithm one might need additional constraints on $\Sigma$ for~\eqref{eq:pop:supp:recovery:def} to hold. These conditions are often called \emph{incoherence conditions} as they measure various sorts of deviations of $\Sigma$ from the identity. Our goal 
is to derive incoherence conditions for the marginal regression to succeed in support recovery over $\Gamma_S$.

In anticipation of the results under sampling, we also introduce a robust version of~\eqref{eq:pop:supp:recovery:def}: $\alg$ succeeds in support recovery \emph{uniformly over $\Gamma_S$ with slack $\delta$} if
\begin{align}\label{eq:pop:supp:recovery:slack:def}
	\alg(\pr_{\beta,\Sigma'}) = \supp(\beta), \quad 
		\forall \pr_{\beta,\Sigma'} \in \Pclin^{p+1},
		\;\; \beta \in \Gamma_S, 
		 \;\; \Sigma'\in \ball_\infty(\Sigma,\beta;\delta/2)
\end{align}
where $\ball_\infty(\Sigma,\beta;\delta) := \{\Sigma':\; \infnorm{(\Sigma' - \Sigma)\beta} \le \delta \}$.  

The population level marginal regression (MR) performs the following operation: 
 (1) Let $\pcor = (\pcor_j) := \cov(X,Y) \in \reals^p$ and sort the coordinates so that  $|\pcor_{i_1}| \ge |\pcor_{i_2}| \ge \dots \ge |\pcor_{i_p}|$.
 (2)  Output $\{i_1,\dots,i_s\}$.
%
The key condition controlling the behavior of population MR is the following:
\begin{defn}[MR incoherence]\label{def:MR:inc}
    A covariance matrix $\Sigma$ satisfies MR incoherence with parameters $R$ and slack $\delta'$ relative to subset $S$, denoted as $\Sigma \in \mri_S(\delta';R)$, if
	\begin{align}
	 \frac{R}{1+R} \|\Sigma_{Sj} \pm \Sigma_{Sk} \|_1 + \frac{\delta'}{ (1+R)} < \Sigma_{jj}\pm \Sigma_{jk},\quad \forall j\in S,\; k\in S^c.
	\label{eq:MR:incoh}
	\end{align}
	Here, $\pm$ signs go together, that is, \eqref{eq:MR:incoh} represents two sets of inequalities.
\end{defn}
As a set of matrices, $\mri_S(\delta';R)$ is decreasing in both its argument $\delta'$ and $R$.  That is,~\eqref{eq:MR:incoh:2} becomes more restrictive as we increase $\delta'$ or $R$. It is also worth nothing that $\mri_S(\delta';R)$ defines a convex subset of the cone of positive semidefinite matrices. 
Our main result regarding support recovery performance of the population MR  is the following:
\begin{thm}[Population MR consistency]\label{thm:MR:main:res} 
	Under  linear model~\eqref{eq:lin:model:pop}, assume that 
	\begin{align}\label{eq:MR:incoh:2}
		\Sigma \in \mri_S\Big( \frac{\delta}{\rho} ; R\Big)
	\end{align}
	for some $\rho, \delta > 0$ and $S \subset [p]$. Then, the following holds:
	\begin{enumerate}
		\item[(a)] For any $R \in [1,\infty]$, the MR incoherence condition~\eqref{eq:MR:incoh:2} is sufficient
		for the population MR to recover the support uniformly over $\Gamma_{S,\rho,R}$ with slack $\delta$ in the sense of~\eqref{eq:pop:supp:recovery:slack:def}.
		\item[(b)] When $R \in [2,\infty]$ or  $\Sigma_{SS}=I$, condition~\eqref{eq:MR:incoh:2} is also necessary.
	\end{enumerate}
\end{thm}

In the special case where $\Sigma_{SS} = I$, one has the following simpler form of~\eqref{eq:MR:incoh}:
\begin{lem}\label{lem:mr:incoh:SigSS:I}
	Assuming that $\Sigma_{SS}=I$, we have $\Sigma \in \mri_S(\delta';R)$  if and only if
	\begin{align}\label{eq:mr:incoh:SigSS:I}
	\norm{ \Sigma_{Sk} }_1  < 
		\frac{1}{R}\Big( 1 - \delta'\Big) + \Big(1- \frac1R\Big) \minnorm{\Sigma_{Sk}},\quad \forall k\in S^c.
	\end{align}
\end{lem}

In light of Theorem~\ref{thm:MR:main:res} and Lemma~\ref{lem:mr:incoh:SigSS:I}, when on-support variables are iid and the non-zero coefficients have zero spread ($R=1$), the Lasso and marginal regression have identical guarantees for support recovery.

 The next lemma shows that for unbounded $R$, uniform recovery is not possible by MR except in the trivial case $\Sigma = I_p$. Note that for $R=\infty$ according to Theorem~\ref{thm:MR:main:res}, MRI is both necessary and sufficient for recovery.
\begin{lem}[$R = \infty$]\label{lem:mr:incoh:R:infinity}
 $\mri_S(\delta';\infty)=\emptyset$ \,for $\delta' \ge 1$ and $=\{I_p\}$ for $\delta' \in [0,1)$. 
\end{lem}


The next proposition shows that if the on-support covariance matrix is close enough to singularity, then  MRI fails to hold. In other words, MR is sensitive to the smallest eigenvalue of $\Sigma_{SS}$.


\begin{prop}\label{prop:small:eigenval}
	Let $\lambda_s^2$ be the smallest eigenvalue of $\Sigma_{SS}$, where $\lambda_s > 0$. Then there exists some $j \in S$ such that for all $k\in S^c$,
	$$
	|\Sigma_{jj} \pm \Sigma_{jk}| \leq \lambda_s \sqrt{s} + \frac{1}{2} \|\Sigma_{Sj} \pm \Sigma_{Sk}\|_1.
	$$
	As a result, if $\lambda_s \le \delta'/(\sqrt{s} (R+1))$ then $\Sigma \not \in \mri_S(\delta';R)$ for any $\delta' \ge 0$ and $R \ge 1$.
\end{prop}

The proofs of Lemma~\ref{lem:mr:incoh:SigSS:I} and~\ref{lem:mr:incoh:R:infinity}, and Proposition~\ref{prop:small:eigenval} appear in Appendix~\ref{sec:rem:proof}. We now consider some examples in which we compare the performance of MR, as controlled by the incoherence introduced in Definition~\ref{def:MR:inc} to that of Lasso. For the incoherence parameter controlling the performance of the Lasso, see Definition~\ref{lasso:incoh}.

We now provide an example which illustrates that MRI could be much more relaxed than PWI, and in the extreme case even match the Lasso incoherence. We assume that the reader is familiar with these two conditions; for details, see~\eqref{eq:LAI:def} and~\eqref{eq:PWI:def} in the appendices. In particular, Lasso incoherence condition, i.e., $\mnorm{\Sigma_{S^c S}^{} \Sigma_{SS}^{-1}}_\infty \le 1-\delta < 1$, is a necessary and sufficient condition for (signed) support recovery by the Lasso. In Appendix~\ref{sec:lasso:incoh}, we provide a self-contained proof of this fact.


\begin{exa}[Pairwise incoherence vs MRI]
\label{ex:pwi_mri_lasso}
	Consider the case where the on-support (i.e., relevant) variables are correlated in groups of size $r$ and the off-support variables are each correlated with at most $r$ of the relevant variables. We compare the bounds that Lasso incoherence~\eqref{eq:LAI:def}, MRI, and PWI impose on the tolerable levels of correlation. We show that the Lasso incoherence, as well as MRI, impose an $O(1/r)$ bound on the cross correlations, MRI additionally imposes an $O(1/r)$ bound on the on-support correlations, and PWI imposes an $O(1/s)$ bound on both the on-support and the cross correlations. When $r=2$, i.e. the case of pairwise correlations, and $R=1$, the Lasso and marginal regression reach identical incoherence conditions, while PWI remains quite restrictive.
	
	More precisely, assume that $\Sigma_{SS}$ is block-diagonal with $b$ blocks $(1-\mu) I_{r_i} + \mu 1_{r_i}1_{r_i}^T,\, i\in[b]$ where $1\leq r:=\max_i r_i < s$ and $\mu \in [-1/(r-1),1]$. Furthermore\footnote{To simplify the calculations, we also assume that for at least one $k$, the support of $\Sigma_{Sk}$ is aligned with one of the blocks in $\Sigma_{SS}$ with size $r$}, assume that the support of $\Sigma_{Sk},\, k\in S^c$ is equal to $\eta 1_r$. Here $1_r \in \reals^r$ is the vector of all ones. The lower bound on $\mu$ is to make $\Sigma_{SS} \succeq 0$. In this case, $\Sigma_{SS}^{-1}$ is block-diagonal with blocks $\frac{1}{1-\mu} \big[ I_{r_i} - \frac{\mu}{1-\mu + \mu r_i} 1_{r_i} 1_{r_i}^T\big],\, i\in[b]$. Then Lasso incoherence~\eqref{eq:LAI:def}, with slack $\delta =0$, is equivalent to
	\begin{align*}
		\max_{k \in S^c} \norm{\Sigma_{SS}^{-1} \Sigma_{Sk}}_1 =  \frac{\eta(r+\mu(3r-4))}{(1-\mu)(1-\mu+r\mu)} \le 1
	\end{align*}
	which holds when $|\eta| \leq (1-\mu) / r$, given $\mu \in (-1/(r-1),1)$. Letting $\gamma := R / (R+1)$, the MR incoherence, with slack $\delta' = 0$, is a subset of $\gamma(1\pm\eta + (r-1 - k)|\mu\pm\eta| + k|\eta|) < 1\pm \eta$ for all $0\leq k \leq r-1$. Thus, the MRI condition is 
	\begin{align*}
		|\eta| < \min \Big\{ \frac{-\mu+\zeta}{1-\zeta}, \frac{\mu+\zeta}{1+\zeta} \Big\},
			\quad |\mu| < \zeta:= \frac1{R(r-1)}.
	\end{align*}
	An interesting case is when $R=1$. We illustrate the conditions for Lasso incoherence, MRI, and PWI in Figure~\ref{fig:comparison:incoherences}. 
\begin{figure}\label{fig:comparison:incoherences}
	\includegraphics[width=.49\textwidth]{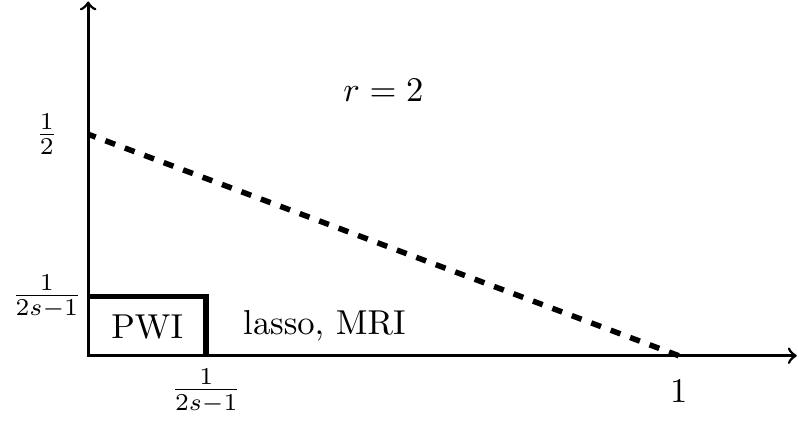}
	\includegraphics[width=.49\textwidth]{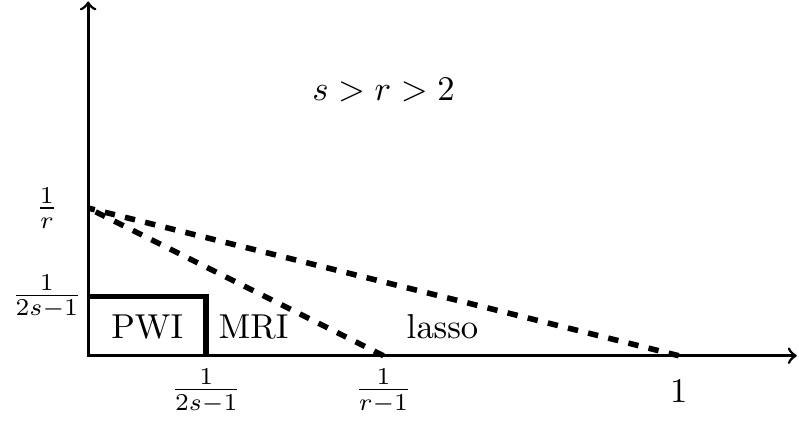}
	\caption{Comparison between the Lasso, MRI, and PWI in Example~\ref{ex:pwi_mri_lasso} for $R=1$. The horizontal axis is the maximum on-support correlation ($\mu$) and the vertical axis the maximum correlation between on and off support variables ($\eta$).}
\end{figure}
	
\end{exa}
	
	The observation that even when $\Sigma_{S S^c} = 0$, MR could fail might seem surprising, but it is  a well-known fact related to the idea of faithfulness in graphical models~\citep{spirtes2000causation}. For specific choices of $\beta$, one could have $\cov(Y,X_j) = 0$ even when $\beta_j \neq 0$ due to the confounding effect of the other variables on the support: $X_i, i \in S\setminus\{j\}$. This type of ``cancellation of correlations'' due to confounding factors has been well-documented in the literature. See for example~\cite{robins2003uniform} and~\cite{Wasserman2009}. Our results, as in Example~\ref{ex:pwi_mri_lasso},  make precise exactly how much confounding from on-support variables can be tolerated by MR before it fails.
	


%% file: mr_sampling.tex
\subsection{Marginal regression under sampling}
\label{sec:mr:sampling}
In this section, we analyze the performance of  MR on a sample of size $n$ from model~\eqref{eq:lin:model:pop}. In particular, we assume
$	\x_ i =(\x_{i1},\dots,\x_{ip}) \sim X$ and $y_i \sim Y$
i.i.d. for $i=1,\dots,n$ where $(X,Y)$ is distributed as in~\eqref{eq:lin:model:pop}. Note that $x_i \in \reals^p$ and $y_i \in \reals$. In addition, we assume that the feature and noise vectors are independent Gaussians: $X \sim N(0,\Sigma)$ and $\eps \sim N(0,\sigma^2 I)$. The sample version of  MR  replaces the population covariance $\pcor= \cov(X,Y)$ with the sample version:
\begin{align*}
	\scor = (\scor_j) = \frac1n \sum_{i=1}^n x_i y_i \in \reals^p.
\end{align*}
For any set $\Gamma \subset \reals^p$ and $t > 0$, let us define
\begin{align}\label{eq:zeta:def}
	\xi(\Sigma; \Gamma,t) 
	:= \sup_{\beta \in \Gamma,\; \norm{\beta}_2 \le t} \sqrt{\beta^T \Sigma \beta}
	= \sup_{\beta \in \Gamma,\; \norm{\beta}_2 \le t} \norm{\Sigma^{1/2} \beta}_2.
\end{align}

\begin{lem}\label{lem:concent:scor}
	Consider model~\eqref{eq:lin:model:pop} with $\beta \in \Gamma_S$ and $\norm{\beta}_2 \le t$, 
	and assume that $\log p / n \le C$ for sufficiently small $C > 0$. Then, with probability at least $1-2p^{-c_1}$,
	\begin{align*}
		\infnorm{\scor - \pcor} \;\le\; \big(\xi(\Sigma;\Gamma_S,t) + \sigma\big) \sqrt{ {c_2 \log p}/{n}}.
	\end{align*}
\end{lem}

	Sample MR succeeds in support recovery whenever $\min_{j \in S} |\scor_j| > \max_{j \in S^c} |\scor_j|$. Combining Lemma~\ref{lem:concent:scor} and Theorem~\ref{thm:MR:main:res}, we have the following guarantee on the support recovery performance of the sample MR:
	\begin{thm}[Sample MR consistency]\label{thm:sample:MR}
		Under the linear model~\eqref{eq:lin:model:pop} with $\beta \in  \Gamma_{S,\rho,R}$ and $\norm{\beta}_2 \le t$, for any $\rho > 0$ and $R \ge 1$,
		the sample MR   recovers the support  with probability at least $1-2p^{-c_1}$ if 
		\begin{align*}
			\Sigma \,\in\, \mri_S\Big(2\big(\xi(\Sigma;\Gamma_S,t) + \sigma\big) \frac1\rho \sqrt{ {c_2 \log p}/{n}};\,R \Big).
		\end{align*}
		
	\end{thm}

Although there could be tighter bounds on $\xi(\Sigma;\Gamma_S,t)$, esp. when $R = 1$, here we consider the general bound  $\xi(\Sigma;\Gamma_S,t) \le \opnorm{\Sigma_{SS}}^{1/2} \min\{t, \sqrt{s} \rho R\} $ which is obtained by noting that
\begin{align*}
\beta^T \Sigma \beta = \beta_S^T \Sigma_{SS} \beta_S = \norm{\Sigma_{SS}^{1/2} \beta_S}_2^2 \le \opnorm{\Sigma_{SS}^{1/2}}^2 \norm{\beta_S}_2^2 = \opnorm{\Sigma_{SS}} \norm{\beta}_2^2,
\end{align*} 
and that $\norm{\beta}_2 \le \min\{t, \sqrt{s} \rho R\}$ for any $\beta \in \Gamma_S$ with $\norm{\beta}_2\le t$. Rewording Theorem~\ref{thm:sample:MR}, using this bound, 
 we have that MR succeeds in support recovery for $\beta \in \Gamma_S$ with $\norm{\beta}_2 \le t$, with probability at least $1-2p^{-c_1}$ if
\begin{align}\label{eq:reword:sample:complexity}
	 \Sigma \,\in\, \mri_S (\delta;\,R ), \quad \text{and} \quad 
	 	n \,\,\gtrsim\,\, \delta^{-2} \big(\sigma^2 +\opnorm{\Sigma_{SS}} \min\{t^2,s \rho^2 R^2\}\big)\,\rho^{-2} \log p.
\end{align}
To observe the typical sample complexity required by~\eqref{eq:reword:sample:complexity}, assume that $\delta = \Omega(1)$, $\opnorm{\Sigma_{SS}} = O(1)$, $\sigma^2 = O(1)$ and either of the following two typical scalings of the parameters hold: (a) $\norm{\beta}_2 \asymp 1$ and $\minnorm{\beta} \asymp 1/\sqrt{s}$ hence $t = O(1)$ and $\rho \asymp 1/\sqrt{s}$ or (b)  $\infnorm{\beta} \asymp \minnorm{\beta} \asymp 1$, that is, $\rho \asymp R \asymp 1$. 
In either of these cases, \eqref{eq:reword:sample:complexity} predicts that $n \gtrsim s \log p$ is sufficient for recovery by MR. This is the well-known minimax scaling of the support recovery problem; see~\cite{wainwright2009information}.

%% file: landscape.tex
\section{Taxonomy of support recovery conditions} \label{sec:landscape}
In this section, we offer a broad perspective on the \emph{truthfulness conditions} governing the performance of several filter and wrapper methods for feature selection. Moreover, we introduce a duality between truthfulness and incoherence conditions for various methods. Our taxonomy further elucidates the strengths and  limitations of MR; relates MR to popular feature selectors including Lasso, (Orthogonal) Matching Pursuit, and SIS; and introduces new conjectures and open questions. 


The connection between a wrapper method and the filter it uses in scoring the covariates at each iteration extends naturally to their corresponding truthful conditions. Here, we consider MR and the related wrapper methods. For the MR, we are interested in the uniform exact recovery in $\Gamma_S$, defined in~\eqref{eq:Gamma:S:def}. Noting that under model~\eqref{eq:lin:model:pop} $\cov(Y,X_j) = \ip{\Sigma_{*j},\beta}$ for any $j$,  the truthfulness condition for MR becomes

\begin{enumerate}[series=truth,label=(F\arabic*)]
\item $\max_{k \in S^c} |\ip{\Sigma_{*k},\beta}| < \min_{j \in S} |\ip{\Sigma_{*j},\beta}|,\quad \forall \beta \in \Gamma_{S,\rho,R}$. \label{truth:mr:r}
\end{enumerate}
We call this \emph{max-min-$R$} condition as it states that the maximum correlation with an off-support covariate must not exceed the minimum correlation with an on-support one. By definition, \ref{truth:mr:r} expresses the requirement that MR fully recovers the support. In Theorem~\ref{thm:MR:main:res}, we proved that MR incoherence is a necessary and sufficient condition for~\ref{truth:mr:r}. Let us separate the case where $R=1$,
\begin{enumerate}[resume=truth,label=(F\arabic*)]
\item $\max_{k \in S^c} |\ip{\Sigma_{*k},\beta}| < \min_{j \in S} |\ip{\Sigma_{*j},\beta}|,\quad \forall \beta \in \Gamma_{S,\rho,1}$.\label{truth:mr:1}
\end{enumerate}
This condition, \emph{max-min-1}, is evidently a weaker condition as it requires uniform recovery over a smaller class of parameters.

For the MP and OMP, following previous literature, one wants conditions for uniform recovery inside $\Gamma_{S,\rho,\infty}$, i.e. all non-zero regression coefficients. The corresponding truthfulness condition is
\begin{enumerate}[resume=truth,label=(F\arabic*)]
\item $\max_{k \in S^c} |\ip{\Sigma_{*k},\beta}| < \max_{j \in S} |\ip{\Sigma_{*j},\beta}|,\quad \forall \beta \in \R^p\backslash\{0\}, \beta_{S^c}=0$.\label{truth:mp:infinity}
\end{enumerate}
This condition, which we call \emph{max-max-$\infty$}, is clearly necessary and sufficient for the first iteration of the MP and OMP to succeed, i.e., for the first step to select a covariate which is truly on-support. The same condition also guarantees correct selection in the remaining iterations; this can be reasoned by considering the correlation between the residual and the remaining covariates and noting that the required condition for selecting a correct covariate is the same as the inequality in~\eqref{truth:mp:infinity} with a different $\beta_S$ from $\R^s\backslash\{0\}$ (ct. Remark~\ref{rem:mp:truth:proof} in Appendix~\ref{sec:other:proofs}). In~\cite{Tropp2004} the author shows one side of the following equivalence:
\begin{lem}\label{lem:lasso:truth:incoh}
	\ref{truth:mp:infinity}	is equivalent to the Lasso incoherence condition:
	$\mnorm{\Sigma_{S^c S}^{} \Sigma_{SS}^{-1}}_\infty  < 1$.
	
\end{lem}
The other side follows easily and, for completeness, we provide a simple proof of both sides in Appendix~\ref{sec:rem:proof}. A consequence of this equivalence is that~\ref{truth:mp:infinity} can serve as a truthfulness condition for the Lasso. Another way to see this is to consider forward-stagewise regression, which is known to be a fast implementation of the Lasso. In this implementation, the covariates with the highest correlation are exploited to update the residual with a fixed step size. The necessity and sufficiency of ~\ref{truth:mp:infinity} for the Lasso thus can be argued by noting the similarity between the MP and forward-stagewise regression. 



It is not clear how~\ref{truth:mp:infinity} connects to~\ref{truth:mr:r} or~\ref{truth:mr:1}. In particular, we do not know if one requires more restrictions on the covariance matrix than the others. However, based on our simulations, we conjecture that~\ref{truth:mr:1} implies~\ref{truth:mp:infinity}, i.e. if for a particular covariance matrix, the marginal regression can recover the support for all $\beta_S \in \{+1, -1\}^s$, then the Lasso is guaranteed to recover the support for all non-zero $\beta$ with support $S$. If this relation holds then the marginal regression could be no stronger than the Lasso in uniform sparse recovery, even under conditions that favor MR ($R=1$).

\begin{conj}
\ref{truth:mr:1} implies \ref{truth:mp:infinity}.
\end{conj}

To give partial evidence of this conjecture, we show that max-min-$R$ when $R=6$ for all subsets $S$ of size $s$ implies \emph{Restricted Isometry Property (RIP)}, a condition known to be sufficient for the close relative of the  Lasso known as the basis pursuit regression; see~\eqref{eq:BP:def} in Appendix~\ref{sec:comparison}. The RIP is defined as follows: (cf.~\cite{HDS})

\begin{figure}\label{fig:truthfulness:incoherence}
	\centering
	\scalebox{.75}{
		\begin{tikzpicture}
		\draw (0,0) rectangle +(2.6,.6);
		\draw (4,0) rectangle +(2.5,.6);
		\draw (8,0) rectangle +(2.8,.6);
		\draw (12,0) rectangle +(2.6,.6);
		
		\draw (.3,-2) rectangle +(2,.6);
		\draw (4.3,-2) rectangle +(1.9,.6);
		\draw (8.2,-2) rectangle +(2.4,.6);
		\draw (12.6,-2) rectangle +(1.4,.6);
		
		\node[above right,label={\large (F1)}] at (0,0) {\Large max-min-$R$};
		\node[above right,label={\large (F2)}] at (4,0) {\Large max-min-1};
		\node[above right,label={\large (F3)}] at (8,0) {\Large max-max-$\infty$};
		\node[above right,label={\large (F4)}] at (12,0) {\Large min-min-$R$};
		
		\node[above] at (1.3,-2.1) {\Large MRI($R$)};
		\node[above] at (5.25,-2.1) {\Large MRI($1$)};
		\node[above] at (9.4,-2.0) {\Large Lasso Inc.};
		\node[above] at (13.3,-2.0) {\Large TBD};
		
		\draw[very thick,->] (2.8,.3) -- +(1,0);
		\draw[very thick,->] (6.8,.3) -- +(1,0);
		\node[above] at (7.3,.3) {\Large ?};
		
		\draw[very thick,<->] (1.3,-.2) -- +(0,-1);
		\draw[very thick,<-] (5.3,-.2) -- +(0,-1);
		\draw[very thick,<->] (9.3,-.2) -- +(0,-1);
		\draw[very thick,<->] (13.3,-.2) -- +(0,-1);
		
		\end{tikzpicture}
	}
	
	\caption{Truthfulness and Incoherence conditions for support recovery.}
\end{figure}
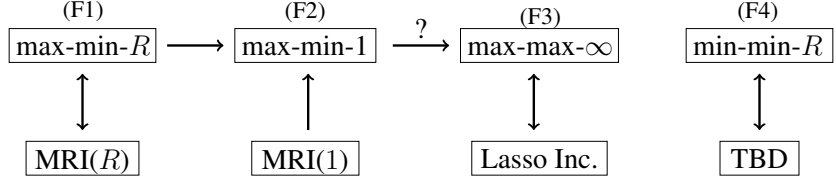

\begin{defn} For  $s \in [p]$,  the covariance matrix $\Sigma$ satisfies a restricted isometry property (RIP) of order $s$ with constant $\delta_s(\Sigma) > 0$ if for all subsets $S$ of size at most $s$, one has $\opnorm{\Sigma_{SS} - I_s} \leq \delta_s(\Sigma).$
	
\end{defn}

A sufficient condition for the exact parameter recovery by the basis pursuit is an RIP condition on $\Sigma$ of order $2s$, namely, $\delta_{2s}(\Sigma) \leq 1/3$.
Let us define $\Gamma_{(s)} := \bigcup_{S :\; |S|\, \le\, s} \Gamma_S$ where $\Gamma_S$ is given in~\eqref{eq:Gamma:S:def}.
An incoherence condition holds over $\Gamma_{(s)}$ iff it holds over $\Gamma_S$ for all subset $S$ of size at most $s$.
The following propositions relates RIP to the MR incoherence derived in Section~\ref{sec:MR:population}. The proofs are deferred to Appendix~\ref{sec:comparison}. 

\begin{prop}\label{prop:MR:RIP}
Assume  $\diag(\Sigma) = 1_p$. If  MR incoherence~\eqref{eq:MR:incoh} holds over $\Gamma_{(s)}$, then 
\begin{align*}
\delta_{2s}(\Sigma) < \frac{(1-\delta')}{R}\frac{2s-1}{s-1}.
\end{align*} 
\end{prop}

This proposition roughly shows that in terms of the strength, the MR incoherence is stronger than RIP, which is in turn comparable in strength to the Lasso incoherence.

Finally, we introduce the truthfulness condition relevant to SIS, which progressively discards irrelevant covariates (conservatively) with the aim of keeping all the relevant ones among the remaining covariates at all times. To achieve this goal, assuming that we want to discard at most $d$ covariates and do so by removing one covariate at a time,  one requires  that the $d$ covariates with least correlations with the target are ``off'' the support, so that we are guaranteed not to discard  any relevant covariate at any stage. Thus, the necessary condition for SIS to achieve uniform recovery over all $\beta \in \Gamma_{S,\rho,R}$ is:
\begin{enumerate}[resume=truth,label=(F\arabic*)]
\item $\min_{k \in S^c} |\ip{\Sigma_{*k},\beta}| < \min_{j \in S} |\ip{\Sigma_{*j},\beta}|,\quad \forall \beta \in \Gamma_{S,\rho,R}$. \label{truth:sis:r}
\end{enumerate}
Formalizing an incoherence condition for~\ref{truth:sis:r} similar to what we have for the other three cases would be of great interest. Note that~\ref{truth:sis:r} is the least stringent of the four conditions we introduced. In other words, we expect~\ref{truth:sis:r} to impose the least constraints on the covariance matrix. As far as we know, little is known about the exact nature of the conditions~\ref{truth:sis:r} imposes.

Figure~\ref{fig:truthfulness:incoherence} summarizes the discussion in this section. It illustrates the relation among the truthfulness conditions and their corresponding incoherence conditions for the MR, (O)MP, Lasso, and the SIS. This taxonomy of truthfulness and incoherence conditions can provide a platform to compare other subclasses of filter and wrapper methods. We hope it can benefit the future research in sparse recovery.


%% file: discussion.tex
\section{Discussion and extensions}\label{sec:discuss}
We studied support recovery performance of marginal regression (MR) and obtained a near complete characterization of the conditions for recovery in terms of the covariance matrix of the features. We introduced parameter $R$ measuring the spread of the coefficients and showed that when $R \ge 2$ or $\Sigma_{SS} = I$, the MR incoherence conditions (Definition~\ref{def:MR:inc}) are necessary and sufficient. We have an example (not mentioned in the paper for brevity) that these conditions are not necessary otherwise (but still sufficient).
An open question is whether the truthfulness of MR implies Lasso incoherence, which is what we conjecture. If true, this settles the question of the dominance of Lasso over MR.
%
Overall, our theory provides a more optimistic view of MR for feature selection and provides some long overdue insights into its strengths and  limitations.
We have also provided a framework to study the truthfulness conditions for general filter and iterative wrapper methods. 

%% file: extension.tex
Finally, we have laid the foundation to extend the MR incoherence condition to non-linear filter methods.
Indeed,  the core results we developed here about the performance of MR can be readily extended to a more general framework. We sketch the steps towards a nonlinear extension here and leave the rigor to a later work. The plan is to show that MR or any other filter that has a semi-norm property  w.r.t. to its arguments (such as tree-based impurity reduction scores when training decision trees) can benefit a performance guarantee similar in nature to that of Theorem~\ref{thm:MR:main:res}, under a general sparse additive model with a sufficiently restricted function class.

Here, we briefly sketch the argument.
Consider a general additive model, in which $Y = \sum_{j\in S} f_j(X_j) + \eps$, with similar assumptions about $X, \eps$ as in~\eqref{eq:lin:model:pop}. Suppose that the function tuple $f=(f_j)_{j\in S}$ is from the following class subclass of $\mathcal F_0 =\{ f:\; \; f'\in L^2(a,\infty) \}$: 
\begin{align}\label{eq:ext:func:reg:assu}
\mathcal{F} := \bigcap_{j,k\in S^c,\;\lambda \in [-1,1]} \big\{f \in \mathcal F_0:\; 
|\ip{Y_{j,k}^\lambda, f'}_{L^2}| \;\ge\; \alpha \norm{Y_{j,k}^\lambda}_{L^2}\, \norm{f'}_{L^2} \big\}
\end{align}
where  $Y_{j,k}^\lambda := \mu_{X_S,X_j} + \lambda \mu_{X_S,X_k}$ with $\mu_{X_S,X_j} = (\mu_{X_i,X_j}, i \in S)$ and  $\mu_{X_i, X_j}(t) := \ex[ X_i 1\{ t \le X_j \} ]$. One can show for functions in this class that $|\cov(X_j+\lambda X_k, f_{i}(X_{i}))|$ is bounded away from zero and also bounded above. Then using the triangle inequality and  sub-linearity of $|\cov(\cdot, Z)|$, one can replicate the argument for Theorem~\ref{thm:MR:main:res}. 
The only obstacle is to show the class of function tuples in~\eqref{eq:ext:func:reg:assu} is non-trivial. We believe this to be true for sufficiently regular classes of functions.  Furthermore, we need to characterize the corresponding class for other semi-norm filter methods, allowing us to replace the Pearson correlation operator $|\cov(\cdot,Z)|$ with a general seminorm filter. We also observe that the role played in our arguments by the spread of the coefficients, the $R$ parameter, will be played by the bounds on the spread of the derivatives of the underlying functions.


%% file: lasso_incoh.tex
\section{Lasso incoherence}\label{sec:lasso:incoh}
In this section, we recall the Lasso incoherence condition and show that it is indeed necessary and sufficient for the Lasso to perform exact ``support plus sign recovery'' at the population level. Though this result is more or less known~\citep{HDS}, there is some nuance to the statement we give here at the population level; in particular, we could not find a result that covers the necessity of Lasso incoherence  (even at the population level) as stated here,
which is part~(a) of Proposition~\ref{prop:lasso:pop}.  We also provide a short self-contained proof of this result for completeness. We then compare the MR incoherence with some well-known incoherences that have appeared in the literature surrounding Lasso and its close relative, the basis pursuit.

\subsection{Lasso at the population level}
Let us start by stating the incoherence condition Lasso is sensitive to:
\begin{defn}\label{lasso:incoh}
	Lasso incoherence (LAI) condition with slack $\delta$, for recovering $S \subset [p]$, is 
	\begin{align}\label{eq:LAI:def}
		\mnorm{\Sigma_{S^c S}^{} \Sigma_{SS}^{-1}}_\infty  \le 1 - \delta
	\end{align}
	and we write $\Sigma \in \lai_S(\delta)$ if this condition holds. We also write $\Sigma \in \lai_S(0+)$ if the condition holds for some $\delta \in (0,1)$.
\end{defn}
An alternative way of writing condition~\eqref{eq:LAI:def} is $\max_{ k \,\in\, S^c} \norm{\Sigma_{SS}^{-1} \Sigma_{Sk}}_1 \le 1- \delta$. It is well known that~\eqref{eq:LAI:def} controls the model selection performance of the Lasso. Consider the population Lasso
\begin{align}\label{eq:lasso:pop}
\bett \in \argmin_{ \beta\, \in\, \reals^p} \, \frac12  \ex (Y - \beta^T X)^2 + \lambda \norm{\beta}_1
\end{align}
where $X \in \reals^p$ and $Y = X^T \beta + \eps \in \reals$ are as in model~\eqref{eq:lin:model:pop}. We say that $\bett$ is \emph{sign selection consistent} if in addition to $\supp(\bett) = \supp(\beta) =: S$, we have $\sign(\bett_S) = \sign(\beta_S)$. Recall the definition of $\Gamma_{S,\rho,R}$ from~\eqref{eq:Gamma:S:def} and note that $\Gamma_{S,0,\infty}$ is the set of all vectors $\beta \in \reals^p$ whose support is contained in $S$. We say that a subset $ \Bc_S \subset \Gamma_{S,0,\infty}$ is \emph{sign rich} if it contains all the sign patterns over $S$, i.e., every $\beta \in \reals^p$ with $\beta_{S^c} = 0$ and $\beta_S \in \bigcup_{\rho > 0} \{-\rho,\rho\}^s$ belongs to $\Bc_S$.

The following proposition shows that $\Sigma \in \lai_S(0)$ is essentially necessary and sufficient for the population Lasso to be sign selection consistent over sign rich sets of parameters. 

%

\begin{prop}\label{prop:lasso:pop}
	Fix some $S \subset [p]$ and assume that $\Sigma_{SS}$ is nonsingular. Let $\Bc_S$ be any sign pattern rich subset of $\Gamma_{S,0,\infty}$. Then, under model~\eqref{eq:lin:model:pop} for $(Y,X)$ with $\beta$ set to $\bets$:
	\begin{itemize}
		\item[(a)] If for every $\bets \in \Bc_S$, the population Lasso in~\eqref{eq:lasso:pop} with input $(Y,X)$ and some $\lambda > 0 $ is sign selection consistent, then $\Sigma \in \lai_S(0)$.
		\item[(b)] If $\Sigma \in \lai_S(0+)$, 
		 then  for every $\bets \in \Bc_S$, the Lasso with input $(Y,X)$ and sufficiently small $\lambda > 0 $, has a unique solution which is sign selection consistent.
	\end{itemize}
\end{prop}

Proposition~\ref{prop:lasso:pop} is more of less colloquially known and can be traced back to the work of~\cite{fuchs2005recovery, tropp2006just, zhao2006model, wainwright2009sharp, Genovese2012} among others. It is not often stated in the population form presented here, and we  give a proof in Appendix~\ref{sec:rem:proof}. Note that part~(a) states something fairly strong: As long as the set of $\beta$s over which we require uniform (sign) selection consistency contains all possible sign patterns, then Lasso incoherence is necessary, no matter how small a value of $\lambda > 0$ we choose. Sign rich sets, for example, include $\Gamma_{S,\rho,R}$ for any $R \in [1,\infty]$.  The interesting aspect of the population level result is that  Lasso is still severely restricted (even with infinite sample size that is) in terms of how much correlation among features it can handle. Contrast this with the ordinary least-squares (or ridge regression), where $\beta$ itself can be recovered exactly at the population level, for any nonsingular covariance matrix $\Sigma$.

%% file: appendices.tex
\section{Proofs}\label{sec:rem:proof}

\subsection{Proof of Theorem~\ref{thm:MR:main:res}}
\label{sec:proofs}
Under model~\eqref{eq:lin:model:pop} $\cov(X,Y) = \ex X(X^T \beta + \eps) = \Sigma \beta$, that is $\cov(X_j,Y) = \ip{\Sigma_{*j},\beta}$ where $\Sigma_{*j}$ is the $j$th column of $\Sigma$.
If follows that MR achieves exact support recovery over $\Gamma_S$ with slack $\delta$ in the sense of~\eqref{eq:pop:supp:recovery:slack:def} if and only if 
\begin{align}\label{eq:mr:recovery:cond:1}
|\ip{\Sigma_{*k},\beta}| + \delta < |\ip{\Sigma_{*j},\beta}|,\quad \forall \beta \in \Gamma_S,\; j \in S,\; k \in S^c.
\end{align}

\begin{lem}\label{lem:a:b:lambda}
	For any $a,b \in \reals$ and $\delta > 0$, we have $|a| - |b| > \delta  \iff  |a-\lambda b| > \delta,\; \forall \lambda \in [-1,1].$

\end{lem}

Using Lemma~\ref{lem:a:b:lambda}, condition~\eqref{eq:mr:recovery:cond:1} is equivalent to $|\ip{\Sigma_{*j}+\lambda \Sigma_{*k}, \beta}| > \delta$ for all $\beta\in \Gamma_S,\; j\in S, \; k\in S^c$ and $\lambda\in [-1,1]$. For a set $\Gamma \subset \reals^p$, we define its \emph{absolute dual} with slack $\delta$ to be 
\begin{align}\label{eq:abs:dual:def}
\Gamma^\dagger= \big\{ \phi \in \reals^p :\; |\ip{\phi, \beta}| > \delta,\; \forall \beta \in \Gamma \big\}.
\end{align}
Thus, exact support recovery by MR over $\Gamma_S$ with slack $\delta$ is equivalent to
\begin{align}\label{eq:mr:recovery:cond:2}
	\Sigma_{*j}+\lambda \Sigma_{*k} \in \Gamma_S^\dagger, \quad \text{for all}\; (j,k) \in S\times S^c, \; \text{and}\;\lambda \in [-1,1].
\end{align}
%
The following technical lemma, proved in Appendix~\ref{sec:rem:proof}, characterizes the absolute dual of the parameters space  $\Gamma_S$:
\begin{lem}\label{lem:Gamma:S:identify}
	The absolute dual of $\Gamma_S$ given in~\eqref{eq:Gamma:S:def} can be written as $\Gamma_S^\dagger = \Gamma_S' \cup \Omega_S$ where 
	\begin{align*}
	\Gamma_S' &:= \Big\{\phi \in \reals^p \,  : \,\, \|\phi_S\|_\infty > \frac{\delta}{\rho (1+R)}+ \frac{R}{1+R} \|\phi_S\|_1 \Big\} , \\
	\Omega_S \subseteq \Gamma_S'' &:= \Big\{\phi \in \reals^p \,  : \,\, \|\phi_S\|_\infty < -\frac{\delta}{\rho (1+R)}+ \frac{1}{1+R} \|\phi_S\|_1 \Big\}.
	\end{align*}
\end{lem}

%

\begin{proof}[Theorem~\ref{thm:MR:main:res}(a)]
	Let $\delta' = \delta/ \rho$. By Lemma~\ref{lem:Gamma:S:identify}, $\Gamma'_S \subset \Gamma_S^\dagger$. Hence, replacing $\Gamma_S^\dagger$ in~\eqref{eq:mr:recovery:cond:2} with $\Gamma_S'$ we get the following sufficient condition for recovery:
	\begin{align}\label{eq:temp:cond:0} 
	\frac{R}{1+R} \|\Sigma_{Sj}+\lambda \Sigma_{Sk} \|_1 +  \frac{\delta'}{1+R} < \infnorm{\Sigma_{Sj}+\lambda \Sigma_{Sk}},
	\quad \forall j\in S, \;k\in S^c,\; \lambda\in [-1,1].
	\end{align}
	We have (see Appendix~\ref{sec:rem:proof} for the proof):
	\begin{lem}\label{lem:infnorm:equality}
		Condition~\eqref{eq:temp:cond:0} implies $\| \Sigma_{Sj}+\lambda \Sigma_{Sk} \|_\infty =  \Sigma_{jj}+\lambda \Sigma_{jk}$ for all $\lambda\in [-1,1]$.
	\end{lem} 
	Therefore, \eqref{eq:temp:cond:0} implies
	\begin{align}\label{eq:temp:cond:1}
	 \frac{R}{1+R} \|\Sigma_{Sj}+\lambda \Sigma_{Sk} \|_1 + \frac{\delta'}{ 1+R} < \Sigma_{jj}+\lambda \Sigma_{jk},\quad \forall j\in S, \,k\in S^c, \,\lambda\in [-1,1].
	\end{align}
	In the other direction, \eqref{eq:temp:cond:1} clearly implies~\eqref{eq:temp:cond:0} noting that $\| \Sigma_{Sj}+\lambda \Sigma_{Sk} \|_\infty  \ge $ $\Sigma_{jj}+\lambda \Sigma_{jk} \ge 0$ for all $|\lambda|\le1$, since by assumption $\Sigma_{jj} = 1$ hence $|\Sigma_{jk}| \le 1$. Since the RHS of~\eqref{eq:temp:cond:1} is convex in $\lambda$ and the LHS is linear, \eqref{eq:temp:cond:1} holds if and only if it holds at the two endpoints $\lambda = -1,1$ which gives~\eqref{eq:MR:incoh} as desired.
\end{proof}

\begin{proof}[Theorem~\ref{thm:MR:main:res}(b), case $R \ge 2$]
	Assume $\rho=1$ without loss of generality. We also write $\Gamma_{S,R} = \Gamma_S$ to emphasize the dependence on $R$.
	For any $\phi\in \reals^p$, let us define 
	\begin{align*}
	\alpha_R(\phi) = \min_{\beta \,\in\, \Gamma_{S,R}} |\ip{\phi, \beta}|
	\end{align*}
	By definition, $\phi$ belongs to $\Gamma_{S,R}^\dagger$
	if and only if $\alpha_R(\phi)> \delta$.  An interesting case is when $R=1$, in which case we use the notation $\alpha(\phi)$ instead of $\alpha_1(\phi)$. The minimization in the case of $R=1$ is over $\Gamma_{S,1}$, which consists of vectors with $\pm \rho$ coordinates on the support. For any $\beta \in \Gamma_{S,1}$, it is easy to see that
	\begin{align}\label{eq:inner:prod:identity}
	\ip{\phi, \beta} = \sum_{j\in S_1} |\phi_j| - \sum_{j \in S_0} |\phi_j| 
	\end{align}
	where $S_1 = \{j\in S \; : \; \sign(\beta_j) = \sign(\phi_j)\}$ and $S_0 = S\backslash S_1$. The relation works both ways, i.e., for any partition of $S$ into $S_1$ and $S_0$, there exists a $\beta\in \Gamma_{S,1}$ that \eqref{eq:inner:prod:identity} holds. 
	
	Let $\{S_0(\phi), S_1(\phi)\}$ be any partition that achieves the minimum for $\phi$ 
	such that $S_1(\phi)$ corresponds to the larger of the two sums, that is
	\begin{align}\label{eq:min:inner:prod:identity}
	\alpha(\phi) = \sum_{j \,\in\, S_1(\phi)} |\phi_j| - \sum_{j \,\in\, S_0(\phi)} |\phi_j| \ge 0.
	\end{align}
	Note that whenever $\phi \neq 0$, we have $|S_1(\phi)| \ge 1$ with $|\phi_j| > 0$ for at least some $j \in S_1(\phi)$. 
	We  adopt the convention of putting the indices of the zero coordinates of $\phi$ in $S_0(\phi)$. Thus, for $\phi \neq 0$, we have $\min_{j \in S_1(\phi)} |\phi_j| > 0 $. 
	As the following lemma states, $\alpha(\phi)$ can be used to approximate $\alpha_R(\phi)$.
	
	\begin{lem}\label{lem:alpha:R:bound}
		If $\phi\in \Gamma_{S,R}^\dagger \setminus \Gamma_{S,R}'$, then \;$\alpha_R(\phi) \leq \max \big\{0, (2-R)\, \alpha(\phi) \big\}$.
	\end{lem}
	As a consequence, $\Gamma_S^\dagger = \Gamma_S'$ whenever $R\geq 2$, completing the proof of part~(b).
\end{proof}

\begin{proof}[Theorem~\ref{thm:MR:main:res}(b), case $\Sigma_{SS} = I$]
	Fix $k\in S^c$. Let $j^*\in \text{argmin}_{j'\in S} |\Sigma_{j'k}|$ and choose $\beta\in \Gamma_{S,\rho,R}$ such that
	\begin{align*}
	\beta_{j^*} = \sign(\Sigma_{j^*k})\,\rho, \quad 
	\beta_{j'}=\sign(\Sigma_{j'k})\,R\rho, \; \;\forall j'\in S\backslash \{j^*\}.
	\end{align*}
	Selection consistency implies that
	\begin{align*}
	R\rho\; \|\Sigma_{Sk}\|_1  - (R\rho - \rho)\; |\Sigma_{Sk}|_{\min} + \delta = |\ip{\Sigma_{Sk}, \beta}| +\delta < |\ip{\Sigma_{Sj^*}, \beta}| =  \rho
	\end{align*}
	using assumption $\Sigma_{SS}=I$ in the last equality. Dividing by $\rho$ and rearraging proves~\eqref{eq:mr:incoh:SigSS:I} which is equivalent to~\eqref{eq:MR:incoh:2}, due to Lemma~\ref{lem:mr:incoh:SigSS:I}.
\end{proof}


\subsection{Proofs of auxiliary lemmas for Theorem~\ref{thm:MR:main:res}}

\begin{proof}[Lemma~\ref{lem:infnorm:equality}]
	Fix $j\in S,\;k\in S^c$ and let $u(\lambda) = \Sigma_{Sj}+\lambda \Sigma_{Sk} \in \reals^s$. Define
	\begin{align*}
	\Omega = \big\{\lambda\in[-1,1]\;:\; \|u(\lambda)\|_\infty = u_j(\lambda) \big\}
	\end{align*}
	where $u_j(\lambda) = \Sigma_{jj} +\lambda \Sigma_{jk}$ is the $j$th component of $u(\lambda)$.
	The set $\Omega$ contains the roots of the function $\lambda \mapsto \infnorm{u(\lambda)} - u_j(\lambda)$, and since this function is continuous, $\Omega$ is a closed set. If we also prove that $\Omega$ is an open set in $[-1, 1]$, then it must be either empty or the whole interval. It is not empty since it contains $0$ (recall the assumption $\Sigma_{jj} = 1$). In the following, we show that $\Omega$ is open in $[-1,1]$. Let
	\begin{align*}
	E := \Big\{\lambda\in[-1,1]\;:\;    \frac{R}{1+R} \|u (\lambda)\|_1 + \frac{\delta'}{1+R} < u_j(\lambda) \Big\} 
	\end{align*}
	Clearly, $E$ is open in $[-1,1]$. We plan to show that $E=\Omega$. First, we have $E\subset \Omega$ since
	\begin{align*}
	\lambda \in E &\implies 
	u_j(\lambda)\, >\, \delta' + R \big[ \|u(\lambda)\|_1 - u_j(\lambda) \big] 
	\,\ge\, \|u(\lambda)\|_1 - |u_j(\lambda)| 
	\,\ge\, \max_{j'\, \in\, S\setminus\{j\}} |u_{j'}| 
	\end{align*}
	implying that $u_j(\lambda) = \infnorm{u(\lambda)}$, that is $\lambda \in \Omega$.
	Assuming \eqref{eq:temp:cond:0}, we also have $\Omega \subset E$. We conclude that $E=\Omega$, hence $\Omega$ is open and should contain all of $[-1,1]$. As a result, \eqref{eq:temp:cond:0} implies \eqref{eq:temp:cond:1} and the proof is complete.
\end{proof}

\begin{proof}[Lemma~\ref{lem:alpha:R:bound}]
	Fix $\phi \in \Gamma_{S,R}^\dagger \setminus \Gamma_{S,R}'$ and let $ S_1 = S_1(\phi)$ and $S_0 = S_0(\phi)$ be as defined in \eqref{eq:min:inner:prod:identity}. According to Lemma~\ref{lem:Gamma:S:identify}, $\phi \in \Gamma_S''$ (and $\phi \neq 0$). This implies that $|S_1| > 1$. Otherwise, $|S_1| = 1$, and $S_1$ should consist of the index of a maximal element of $\phi$ in absolute value,  hence
	\begin{align*}
	\alpha(\phi) = \|\phi_S\|_\infty - (\|\phi_S\|_1 - \|\phi_S\|_\infty) = 2 \|\phi_S\|_\infty - \|\phi_S\|_1 < 0
	\end{align*}
	a contradiction.
	The last inequality is a consequence of $\phi \in \Gamma_S''$.
	
	\medskip
	Now, let $c=\sum_{j \in S_0} |\phi_j|$. Then, $\alpha(\phi) = \sum_{j \in S_1} |\phi_j| - c$. Let $j^* = \text{argmin}_{j'\in S_1} |\phi_{j'}|$ and note that $|\phi_{j^*}| > 0$ by construction. 
	Then,
	\begin{align}\label{eq:temp:phi:j:star:upper}
	|\phi_{j^*}| \leq \frac{1}{|S_1|}\; \sum_{j \in S_1} |\phi_j| \leq \frac{1}{2}\; (\alpha(\phi)+c)
	\end{align}
	where we have used $|S_1| > 1$. 
	Now, let $S_1' = S_1 \backslash \{j^*\},\; S_0' = S_1 \cup \{j^*\}$. We have
	\begin{align*}
	\big| \alpha(\phi) - 2 |\phi_{j^*}| \big| = \Big|\sum_{j\in S_1'} |\phi_j| - \sum_{j\in S_0'} |\phi_j| \Big| \ge \alpha(\phi)
	\end{align*}
	where the inequality is by the optimality of $\{S_0,S_1\}$ partition. Since $|\phi_{j*}| > 0$, this implies $\alpha(\phi) \leq |\phi_{j^*}|$.
	%
	Combining with~\eqref{eq:temp:phi:j:star:upper}, we have 
	$\alpha(\phi) \leq c$.
	
	For any $1 \leq \gamma \leq R$, define $\beta^{(\gamma)} \in \Gamma_{S,R}$ such that $\beta^{(\gamma)}_j = \sign(\phi_j) \big[ 1\{j\in S_1\} -\gamma\; 1\{j\in S_0\} \big]$, for all $j\in S$. Then,
	\begin{align*}
	\alpha_R(\phi) \leq  \min_{1\leq \gamma \leq R} |\ip{\phi, \beta^{(\gamma)}}| 
	= \max \Big\{0,\; \sum_{j\in S_1} |\phi_j| - R \sum_{j \in S_0} |\phi_j| \Big\}  
	\end{align*}
	where the equality follows since $\ip{\phi, \beta^{(\gamma)}}$ is a decreasing and continuous function of $\gamma$. 
	But
	\begin{align*}
	\sum_{j\in S_1} |\phi_j| - R \sum_{j \in S_0} |\phi_j| = \alpha(\phi) - (R-1)\;c \leq  (2-R)\;\alpha(\phi)
	\end{align*}
	where we have used $\alpha(\phi) \leq c$. The proof is complete.
\end{proof}

\subsection{Proof of Theorem~\ref{thm:sample:MR}}
	Let us write $\mri_S(\delta') = \mri_S(\delta';R)$ for simplicity. The condition for the sample MR to recover the support is $|\scor_j| > |\scor_k|$ for all $j \in S$ and $k \in S^c$. By triangle inequality, $|\scor_j| \ge |\pcor_j| - \infnorm{\scor - \pcor}$ and $|\scor_k| \le |\pcor_k| + \infnorm{\scor - \pcor}$. Thus sample MR is consistent if $|\pcor_j| > 2 \infnorm{\scor - \pcor} +  |\pcor_k|$ for all $j \in S$ and $k \in S^c$ which is equivalent to population MR consistency with slack $\delta := 2 \infnorm{\scor - \pcor}$. Theorem~\ref{thm:MR:main:res} thus gives the following sufficient condition
	\begin{align*}
	\Sigma \;\in\; \mri_S\Big(\frac{2 \infnorm{\scor - \pcor}}\rho\Big) \;\le \;
	\mri_S \Big( 2\big(\xi(\Sigma;\Gamma_S,t) + \sigma\big) 
	\frac1\rho \sqrt{   \frac{c_2 \log p}{n}} \;\Big)
	\end{align*}
	where the second inequality holds with probability $\ge 1-2p^{-c_1}$ by Lemma~\ref{lem:concent:scor} and the fact that $\mri_S(\delta'_1) \subset \mri_S(\delta'_2)$ whenever $\delta'_1 \le \delta'_2$.  

\subsection{Other proofs}\label{sec:other:proofs}
\begin{proof}[Lemma~\ref{lem:mr:incoh:SigSS:I}]
	Fix $j\in S,\; k\in S^c$. Assumption~\eqref{eq:mr:incoh:SigSS:I} implies $R \norm{ \Sigma_{Sk} }_1 < 1 -\delta' + (R-1) |\Sigma_{jk}|$ or equivalently
	\begin{align*}
	R \sum_{j' \in S\backslash \{j\}} |\Sigma_{j'k}|  < 1 - \delta' - |\Sigma_{jk}|
	\end{align*}
	Replacing $- |\Sigma_{jk}|$ with $\pm \Sigma_{jk}$, adding $R (1 \pm \Sigma_{jk})$ to both sides, and using $\Sigma_{j'j} = 1\{j' = j\}$,
	\begin{align*}
	R \| \Sigma_{Sj} \pm \Sigma_{Sk} \|_1  < (R+1)(1 \pm \Sigma_{jk}) -\delta' ,
	\end{align*}
	which gives the desired result after some algebra. The only if part holds as the previous argument is reversible. 
\end{proof}

\begin{proof}[Lemma~\ref{lem:mr:incoh:R:infinity}]
Fix $\delta' \geq 0$. Assume $\Sigma \in \mri_S(\delta';R)$, for all $R < \infty$, or equivalently
\begin{align*}
R \sum_{j'\in S\backslash \{j\}} |\Sigma_{j'j} \pm \Sigma_{j'k} | + \delta' < \Sigma_{jj}\pm \Sigma_{jk},\quad \forall j\in S,\; k\in S^c, R\in [1, \infty).
\end{align*}
The feasibility of these conditions for every $R\in [1, \infty)$ requires that
\begin{align*}
\sum_{j'\in S\backslash \{j\}} |\Sigma_{j'j} \pm \Sigma_{j'k} | = 0,\quad \forall j\in S,\; k\in S^c
\end{align*}
which implies that $\Sigma_{j'k}=\Sigma_{j'j}=0$ for all $j,j'\in S,\; k\in S^c$  that $j\not= j'$. Therefore, $\Sigma_{S^cS}=0$ and $\Sigma_{SS}=I$. Moreover, $\delta'$ must be less than one. This proves the desired result.
\end{proof}

\begin{proof}[Proposition~\ref{prop:small:eigenval}]
	Let $\mu$ be a unit norm eigenvector of $\Sigma_{SS}$ associated with eigenvalue $\lambda_s^2$. Then,
	$
	\var(\mu^T X_S) = \mu^T \Sigma_{SS}\, \mu = \lambda_s^2 
	$
	and as a result,
	$$
	| \mu^T \Sigma_{Si} | = |\text{cov}(\mu^T X_S, X_i)| \leq \big( \var(\mu^T X_S) \var(X_i) \,\big)^{1/2} \leq \lambda_s, \quad i\in [p]
	$$
	Pick $j = \text{argmax}_{j' \in S} |\mu_{j'}|$ and let $S' := S \setminus \{j\}$. For any $k\in S^c$,
	\begin{align*}
	\big|\mu_j \Sigma_{jj} \pm \mu_j \Sigma_{jk} \big|  - \big| \mu_{S'}^T \Sigma_{S'j} \pm \mu_{S'}^T \Sigma_{S'k} \big|
	\stackrel{(i)}{\le} \big| \mu^T \Sigma_{Sj}\big| + \big|\mu^T \Sigma_{Sk}\big| \le 2\lambda_s
	\end{align*}
	where $(i)$ is by the triangle inequality $|a\pm c| - |b\pm d| \leq |a+b| + |c+d|$.
	Rearranging and noting that $|\mu_j| >  0$ (since $\mu \neq 0$), we have with $\nu_{S'} := \mu_{S'} / |\mu_j|$,
	\begin{align*}
	\big|\Sigma_{jj} \pm \Sigma_{jk} \big|  
	&\le \frac{2\lambda_s}{|\mu_j|} + \big| \ip{\nu_{S'}, \Sigma_{S'j} \pm  \Sigma_{S'k}} \big| \\
	&\le \frac{2\lambda_s}{|\mu_j|} + 
	\infnorm{\nu_{S'}} \norm{ \Sigma_{S'j} \pm  \Sigma_{S'k}}_1 \le \frac{2\lambda_s}{|\mu_j|} + \norm{ \Sigma_{S'j} \pm  \Sigma_{S'k}}_1
	\end{align*}
	since $\infnorm{v_{S'}} \le 1$ which holds by the particular choice of $j$.
	%
	Adding $|\Sigma_{jj} \pm \Sigma_{jk}|$ and dividing by 2,
	$$
	|\Sigma_{jj} \pm \Sigma_{jk}| \leq \frac{\lambda_s}{\|\mu\|_\infty} + \frac{1}{2} \|\Sigma_{Sj} \pm \Sigma_{Sk}\|_1.
	$$
	Since $\norm{\mu}_2 = 1$, we have $\infnorm{\mu} \ge 1/\sqrt{s}$ which gives the desired inequality.
	The last assertion of the proposition follows form  the inequality by noting that $R/(R+1)\ge 1/2$.
\end{proof}


\begin{proof}[Lemma~\ref{lem:concent:scor}]
	We have $Y \sim N(0,\beta^T \Sigma \beta + \sigma^2)$ and $X_j \sim N(0,1)$. It follows that $\sgnorm{Y} \lesssim \sqrt{\beta^T \Sigma \beta + \sigma^2} \le \sqrt{\beta^T \Sigma \beta} + \sigma$ and  $\sgnorm{X_j} \lesssim 1$. Then, by~\cite[Lemma 2.7.7]{vershynin2016high} $Y X_j$ is sub-exponential and
	\begin{align*}
	\senorm{Y X_j} \,\le\, \sgnorm{Y} \sgnorm{X_j} 
	\,\lesssim\, \sqrt{\beta^T \Sigma \beta} + \sigma 
	\,\le\,	\xi(\Sigma;\Gamma_S,t) + \sigma
	\,:=\, \xi'
	\end{align*}
	using $\beta \in \Gamma_S$ and $\norm{\beta}_2 \le t$ by assumption and definition~\eqref{eq:zeta:def}.
	By centering~\cite[Exercise 2.7.10]{vershynin2016high}, we have the same bound, up to constants, for $\senorm{Y X_j - \ex[Y X_j]} \lesssim \senorm{Y X_j} \lesssim \xi'$.
	We note that 
	$\scor_j - \pcor_j = \frac1n \sum_{i=1}^n  y_i  x_{ij}- \ex[y_i x_{ij} ]$
	which is an average of iid centered sub-exponential variables distributed as $Y X_j - \ex[Y X_j]$. Bernstein inequality for sub-exponential variables~\cite[Theorem 2.8.1]{vershynin2016high} implies that for any fixed $j \in [p]$,
	\begin{align*}
	\pr \big( |\scor_j - \pcor_j| \ge \xi' \tau \big) 
	\le 2 \exp \big[ {-c} \,n \min \big( \tau^2 ,  \tau \big) \big] ,\quad \tau \ge 0.
	\end{align*}
	Applying union bound over $j=1,\dots,p$, 
	\begin{align*}
	\pr \big( \infnorm{\scor- \pcor} \ge \xi' \tau \big) 
	\le 2 p\exp \big[ {-c} \,n \min \big( \tau^2 ,  \tau \big) \big] ,\quad t \ge 0.
	\end{align*}
	Taking $\tau = \sqrt{(1+c_1) \log p / (cn)} \le 1$, where the inequality holds by the assumption that $\log p /n$ is sufficiently small, we obtain the result.
\end{proof}

\begin{proof}[Lemma~\ref{lem:a:b:lambda}]
	The result follows from the following identity
	\begin{align}\label{eq:a:b:lam:identity}
	\inf_{\lambda \,\in\, [-1,1]} |a - \lambda b| = \big( |a| - |b| \big)_+  \quad \forall a,b \in \reals,
	\end{align}
	where $x_+ := \max(x,0) = x 1\{x > 0\}$ is the positive part of $x$. To see  identity~\eqref{eq:a:b:lam:identity} note that it holds trivially for $b=0$. Now, assume $b \neq 0$, and let $\alpha = a/b$. Then,
	\begin{align*}
	\inf_{\lambda \,\in\, [-1,1]} |\alpha - \lambda | =  \inf_{\lambda \,\in\, [-1,1]} ||\alpha| - \lambda |  = 
	\begin{cases}
	0 & |\alpha | \le 1 \\
	|\alpha| - 1 & |\alpha | >1
	\end{cases}.
	\end{align*}
	Multiplying by $|b|$ gives~\eqref{eq:a:b:lam:identity}.
\end{proof}


\begin{proof}[Lemma~\ref{lem:Gamma:S:identify}]
	First, we prove $\Gamma_S' \subseteq \Gamma_S^\dagger$. Take any $\beta \in \Gamma_S$ and $\phi \in \Gamma_S'$.
	Let $j \in \text{argmax}_{j'\in S}\; |\phi_{j'}|$. Then,
	\begin{align*}
	|\ip{\phi, \beta}| &= |\ip{\phi_S, \beta_S}| \\
	&\geq |\phi_{j}|\; |\beta_{j}| - \big| \sum_{j'\in S\backslash \{j\}} \phi_{j'}\;\beta_{j'}\big| \\
	&\geq |\phi_{j}|\; |\beta_{j}| - \sum_{j'\in S\backslash \{j\}} |\phi_{j'}|\; |\beta_{j'}| \\
	&\geq \|\phi_S\|_\infty \; \min_{j\in S} |\beta_j| - (\|\phi_S\|_1 - \|\phi_S\|_\infty)\; \|\beta_S \|_\infty \\
	&\overset{(i)}{\geq} \|\phi_S\|_\infty \; \min_{j\in S} |\beta_j| - R\;(\|\phi_S\|_1 - \|\phi_S\|_\infty)\; \min_{j\in S} |\beta_j| \\
	&\overset{(ii)}{>} \frac{\delta}{\rho} \; \min_{j\in S} |\beta_j|  \\
	&\overset{(iii)}{\geq} \delta,
	\end{align*}
	where (i) and (iii) are implied by $\beta \in \Gamma_S$ and (ii) by $\phi \in \Gamma_S'$. Notice that the second to last inequality is strict.
	
	In order to prove the second part of lemma, fix some $\phi \not\in \Gamma_S' \cup \Gamma_S''$. We show that $\phi \not \in \Gamma_S^\dagger$, by constructing some $\beta \in \Gamma_S$ such that $|\ip{\phi, \beta}| \leq \delta$. Let $\delta' = \delta/ \rho$. With some algebra, we have
	\begin{align*}
	R^{-1}( \|\phi_S\|_\infty	- \delta')
		\; \leq\; \|\phi_S\|_1 - \|\phi_S\|_\infty \;\leq\;
		 R \,\|\phi_S\|_\infty + \delta'.
	\end{align*}
	Let us define the following mutually exclusive intervals:
	\begin{align*}
	I_1 &= \big[ R^{-1} (\|\phi_S\|_\infty - \delta'),\; R^{-1} \|\phi_S\|_\infty \big], 
	\quad I_2 = \big( R^{-1} \|\phi_S\|_\infty,\; \|\phi_S\|_\infty \big] \\
	I_3 &= \big( \|\phi_S\|_\infty,\; R\; \|\phi_S\|_\infty \big] 
	\quad I_4 = \big( R\; \|\phi_S\|_\infty,\; R\; \|\phi_S\|_\infty + \delta' \big].
	\end{align*}
	We consider four cases corresponding to $\|\phi_S\|_1 - \|\phi_S\|_\infty \in I_i$, for  $i=1,2,3,4$. We construct $\beta^1,\beta^2,\beta^3,\beta^4 \in \reals^p$ such that whenever $\|\phi_S\|_1 - \|\phi_S\|_\infty \in I_i$ then $\beta^i\in \Gamma_S$ and $|\ip{\phi, \beta^i}| \leq \delta'$. 
	Let us proceed with the construction.
	We set $\beta^i_{S^c}=\bar{0}_{p-s}$ for all $i=1,2,3,4$ and let $j \in \text{argmax}_{j'\in S}\; |\phi_{j'}|$. Define $\beta^i_j$ as follows:
	\begin{align*}
	\beta^i_j =  \sign(\phi_j)\; \left\{ \begin{array}{ll}
	1 & : i=1 \\
	\|\phi_S\|_\infty / (\|\phi_S\|_1 - \|\phi_S\|_\infty) & :i=2 \\
	-1 & : i=3 \\
	-R & : i=4
	\end{array} \right.
	\end{align*} 
	and for $j'\in S\backslash \{j\}$,  let
	\begin{align*}
	\beta^i_{j'} =  \sign(\phi_{j'})\; \left\{ \begin{array}{ll}
	-R & : i=1 \\
	-1 & :i=2 \\
	(\|\phi_S\|_1 - \|\phi_S\|_\infty) / \|\phi_S\|_\infty  & : i=3 \\
	1 & : i=4
	\end{array} \right.
	\end{align*}
	We leave it to the reader to verify that for $i=1,4$ we have $0 \leq \ip{\phi,\beta^i} \leq \delta'$ and for $i= 2, 3$, we have $\ip{\phi,\beta^i}=0$. Furthermore, one should check that $\beta^2, \beta^3 \in \Gamma_S$.
\end{proof}
\begin{rem}\label{rem:mp:truth:proof}
The proof is by induction. Suppose the selected covariates after some iterations are $X_j,\, j\in \widehat{S}$ and $\widehat{S} \subset S$. Also assume the residual is 
\begin{align*}
R = Y - \sum_{j\in \widehat{S}} \gamma_j X_j
\end{align*}
We don't make any assumption about the origin of $\gamma_j$ so the rest of proof works for MP, OMP, and forward-stagewise regression. Denote $\gamma = (\gamma_j)_{j\in\widehat{S}}$. Now let us compute the covariance of $R$ and each covariate, $X_j,\, j\in [p]$:
\begin{align*}
\cov(R, X_j) &= \cov(Y,X_j) - \sum_{j'\in \widehat{S}} \gamma_{j'} \cov(X_{j'}, X_j) \\
&= \ip{\Sigma_{*j}, \beta} - \ip{\Sigma_{\widehat{S}j}, \gamma} \\
&= \ip{\Sigma_{*j}, \beta - \bar{\gamma}}
\end{align*}
where $\bar{\gamma}\in \R^p$ satisfies $\bar{\gamma}_{\widehat{S}} = \gamma$ and $\bar{\gamma}_{\widehat{S}^c}=0$. Since $\beta-\bar{\gamma}$ has support $S$ and it is non-zero (unless the residual is zero and we stop the iteration), the condition~\ref{truth:mp:infinity} guarantees the next selection would be also from the support.
\end{rem}
\begin{proof}[Lemma~\ref{lem:lasso:truth:incoh}]
If $\Sigma_{SS}$ is singular, neither of~\ref{truth:mp:infinity} or Lasso incoherence holds. Therefore assume $\Sigma_{SS}$ is non-singular,
\begin{align*}
\infnorm{\Sigma_{S^cS} \beta} < \infnorm{\Sigma_{SS} \beta},\quad \forall \beta \in \R^s \{0\} &\iff \infnorm{\Sigma_{S^cS} \Sigma_{SS}^{-1} \beta} < \infnorm{\Sigma_{SS} \Sigma_{SS}^{-1} \beta},\quad \forall \beta \in \R^s \{0\} \\
&\iff \opinfnorm{\Sigma_{S^cS}\Sigma_{SS}^{-1}} < 1
\end{align*}
The first equivalence holds since the linear operator defined by $\Sigma_{SS}$ is a bijection on $\R^s\backslash \{0\}$ and the second equivalence holds by definition of the operator norm $\opinfnorm{}$.
\end{proof}
%

\begin{proof}[Proposition~\ref{prop:lasso:pop}]
	Let us write $\Sigma_{xy} = \ex[XY]$ and $\Sigma_{xx} = \Sigma = \ex[XX^T]$.
	The Lasso solution can be written as
	\begin{align*}
	\bett \in \argmin_{\beta} \frac12 \beta^T \Sigma_{xx} \beta - \beta^T \Sigma_{xy} + \lambda \norm{\beta}_1
	\end{align*}
	The optimality conditions are obtained by requiring that zero belongs to the subdifferential of the objective, i.e.,  $\Sigma_{xx} \bett - \Sigma_{xy} + \lambda u$ where $ u \in \partial \norm{\bett}_1$. Alternatively, $u = \sign(\bett)$ where $\sign$ should be interpreted as a generalized sign vector (i.e., for the scalar version $\sign(x) \in [-1,1]$ whenever $x = 0$).  Under model~\eqref{eq:lin:model:pop} with $\beta = \bets$, we have $\Sigma_{xy} = \Sigma \bets$. The optimality conditions are given by
	\begin{align}
	\Sigma(\bett - \bets) + \lambda u =0 , \quad u = \sign(\bett).
	\end{align}
	Let us write $\Delta =\bett - \bets$.  Consider part~(a) first: Assume that $\bett$ has the correct support, so that $\bett_{S^c} = 0$, and since $\bets_{S^c}$, we have $\Delta_{S^c} = 0$. Partitioning over $S$ and $S^c$, we have
	\begin{align} \label{eq:pop:lasso:optim}
	\begin{pmatrix}
	\Sigma_{SS}  & \Sigma_{SS^c} \\	\Sigma_{S^c S}  & \Sigma_{S^c S^c} 
	\end{pmatrix}
	\begin{pmatrix}
	\Delta_S \\ 0
	\end{pmatrix} + \lambda
	\begin{pmatrix}
	u_S \\ u_{S^c}
	\end{pmatrix} = 0
	\end{align}
	that is, $ \Delta_S = - \lambda \Sigma_{SS}^{-1} u_S$ and $\Sigma_{S^c S} \Delta_S + \lambda u_{S^c} = 0$. Substituting the expression for $\Delta_S$ from the first equation into the second, we obtain (using $\lambda > 0$)
	\begin{align}\label{eq:pop:lasso:uSc}
	u_{S^c} = \Sigma_{S^c S}  \Sigma_{SS}^{-1} u_S.
	\end{align} 
	Since $u_{S^c}$ is a generalized sign vector, we have $\infnorm{u_{S^c}} \le 1$. Assuming that $\bett_S$ has the correct sign, we get $u_S = \sign(\bett_S) = \sign(\bets_S)$. As $\bets_S$ varies in $\Bc_S$,  $\sign(\bets_S)$ takes all possible values in $\{\pm 1\}^s$. Thus, under the assumption of part~(a), we have
	\begin{align*}
	\sup_{u_S \,\in\, \{\pm 1\}^s}\infnorm{ \Sigma_{S^c S}  \Sigma_{SS}^{-1}\, u_S } \le 1.
	\end{align*}
	The LHS is equal to $\mnorm{\Sigma_{S^c S}  \Sigma_{SS}^{-1}}_\infty$ completing the proof of part~(a). (Note that the maximum of a convex function over a set is equal to its maximum over the convex hull of that set,  hence $\sup_{u_S \,\in\, \{\pm 1\}^s} f(u_S) = \sup_{u_S \in \ball_\infty} f(u_S)$ for any convex $f$.)
	
	\smallskip
	For part~(b), consider a candidate $\bett$ with $\bett_{S^c} = 0$ and $\bett_S = \bets_S - \lambda \Sigma_{SS}^{-1} u_S$. Choose $u_S$ such that it satisfies
	\begin{align*}
	u_S = \sign(\bets_S - \lambda \Sigma_{SS}^{-1} u_S)
	\end{align*}
	which always has a solution for sufficiently small $\lambda > 0$. In fact, for sufficiently small $\lambda$, we obtain $u_S = \sign(\bets_S)$. Define $u_{S^c}$ as in~\eqref{eq:pop:lasso:uSc}. This is a valid choice since $\infnorm{u_{S^c}} \le \mnorm{\Sigma_{S^c S}  \Sigma_{SS}^{-1}}_\infty \infnorm{u_S} < 1$ using the assumption of part~(b) and that $\infnorm{u_S} \le 1$. It follows that this dual vector is strictly feasible.  The constructed pair $(\bett,u)$ is then feasible and satisfies the optimality condition~\eqref{eq:pop:lasso:optim}, hence it is an optimal primal-dual pair; in addition strict dual feasibility implies uniqueness of the primal solution. Hence, constructed $\bett$ is the unique solution of the Lasso, with correct support and correct sign $\sign(\bett_S) = u_S = \sign(\bets_S)$.
\end{proof}

\input{comparison}

\begin{proof}[Proof of Proposition~\ref{prop:MR:PW}]
	Taking $j\in S$ and $k\in S^c$,
	\begin{align*}
	\sum_{j' \in S\backslash\{j\}} \big| \Sigma_{j'j} \pm \Sigma_{j'k} \big| &\leq 2(s-1)\; \delta_\text{PW}(\Sigma)\\
	&= 2(s-1)\;\delta_\text{PW}(\Sigma) +  \frac{1}{R}\; \delta_\text{PW}(\Sigma) -  \frac{1}{R}\; \delta_\text{PW}(\Sigma) \\
	& < \frac{1-\delta'}{R} -  \frac{1}{R}\; \delta_\text{PW}(\Sigma) \\
	& \leq \frac{1}{R} \; \big( \Sigma_{jj} \pm \Sigma_{jk}  \big) - \frac{\delta'}{R}
	\end{align*}
	where the first inequality is by assumption and the second uses $\Sigma_{jj}=1$. The MR incoherence~\eqref{eq:MR:incoh}  follows by adding $\big| \Sigma_{jj} \pm \Sigma_{jk}  \big| = \Sigma_{jj} \pm \Sigma_{jk}$ to both sides and multiplying by $R/(R+1)$.
\end{proof}

\begin{proof}[Proof of Proposition~\ref{prop:MR:RIP}]
	Pick some $\St \subset [p]$ with size $2s$. Choose a balanced partition, $S_0, S_1$, of $\St$, i.e. $|S_0| = |S_1|$. The MR incoherence condition holds for both $S=S_1$ and $S=S_2$:
	\begin{align}\label{eq:temp:cond:2}
	\|\Sigma_{S_ij} \pm \Sigma_{S_ik}\|_1 
	< \frac{1+R}{R} (1 \pm \Sigma_{jk}) - \frac{\delta'}{R} ,\quad j\in S_i,\; k\in S_{1-i},\; i \in \{0,1 \}.
	\end{align}
	Fix some $i \in \{0,1\}$, $j \in S_i$ and $k \in S_{1-i}$. 
	By the convexity of $\ell_1$ norm,
	\begin{align}\label{eq:Sigma:Si:j}
	\|\Sigma_{S_i j}\|_1 
	\;\leq\; \frac12 \big( \|\Sigma_{S_ij} + \Sigma_{S_ik}\|_1 + \|\Sigma_{S_ij} - \Sigma_{S_ik}\|_1 \big) 
	\;<\; \frac1R (1+R-\delta').
	\end{align}
	Let $\Delta = \Sigma - I_{p}$. Note that $\Delta_{S_i S_i} = \Sigma_{S_i S_i} - I_s$, hence 
	\begin{align}\label{eq:Delta:Si:j}
	\|\Delta_{S_i j}\|_1 = \norm{\Sigma_{S_i j}}_1 - 1 < (1-\delta')/R
	\end{align}
	by~\eqref{eq:Sigma:Si:j} and the assumption $\diag(\Sigma_{S_i S_i}) = 1_s$.  Let $e_j$ be the $j$th basis vector of $\reals^s$. Then, $\Delta_{S_ij} \pm \Delta_{S_ik} = \Sigma_{S_ij} \pm \Sigma_{S_ik} - e_j$, hence 
	\begin{align*}
	\norm{\Delta_{S_ij} \pm \Delta_{S_ik}}_1 
	&= |\Sigma_{jk}| + \sum_{j' \in S \setminus \{j\} } | \Sigma_{j'j} \pm \Sigma_{j'k}| \\
	&\le |\Sigma_{jk}| - (1 \pm \Sigma_{jk}) + \vnorm{\Sigma_{S_ij} \pm \Sigma_{S_ik}}_1 \\
	&< \frac{1}{R} (1 \pm \Delta_{jk}) + |\Delta_{jk}| - \frac{\delta'}{R}
	\end{align*}
	using $\Sigma_{jj} = 1$, \eqref{eq:temp:cond:2} and $\Sigma_{jk} = \Delta_{jk}$.
	
	Using a convexity argument as in~\eqref{eq:Sigma:Si:j}, we obtain $\|\Delta_{S_ik}\|_1  < (1-\delta')/R +  |\Delta_{jk}|$. Taking the sum over $j \in S_i$ and rearranging, we have
	\begin{align}\label{eq:Delta:Si:k}
	\|\Delta_{S_i k}\|_1 <  \frac{s(1-\delta')}{R\;(s-1)}.
	\end{align}
	Since~\eqref{eq:Delta:Si:j} and~\eqref{eq:Delta:Si:k} hold for any $j \in S_i$ and $k \in S_{1-i}$, we have shown that
	every column of $\Delta_{S_i S_i}$ and $\Delta_{S_i S_{1-i}}$, for $i=0,1$, has $\ell_1$ norm bounded by $(1-\delta')/R$ and $s(1-\delta')/R(s-1)$, respectively. It follows that
	\begin{align*}
	\opnorm{\Delta_{\St \St}} \le \mnorm{\Delta_{\St \St}}_1 = \max_{\ell \in \St} \norm{\Delta_{\St \ell}}_1 \le \frac{1-\delta'}{R} + \frac{s(1-\delta')}{R(s-1)}
	\end{align*}
	where the first inequality is well-known (the $\ell_1$ operator norm bounds the $\ell_2$ operator norm for symmetric matrices) and the equality as well: the $\ell_1$ operator norm is the maximum absolute column sum. Since $\St$ was an arbitrary subset of size $2s$, the proof is complete.
	%
\end{proof}


%% file: comparison.tex
\subsection{Comparison with other incoherence conditions}
\label{sec:comparison}

Let us recall some other incoherence conditions that are often considered when studying the \emph{basis pursuit}, a close relative of Lasso, which at population level solves the optimization problem:
\begin{align}\label{eq:BP:def}
	\min_{\beta'} \norm{\beta'}_1 \quad \text{subject to }\; \Sigma \beta' = \cov(X,Y). 
\end{align}
The following  incoherence condition is well-known~\cite{HDS}:
\begin{defn}\label{defn:PWI:def} Pairwise incoherence (PWI) parameter of the covariance matrix $\Sigma$ is defined as
\begin{align}\label{eq:PWI:def}
\dpw(\Sigma) := \max_{1\leq i \leq j \leq p}|\Sigma_{ij} - 1\{i = j\}|. 
\end{align}
\end{defn}

%

It is known that if $\Sigma$ has pairwise incoherence  $\delta_\text{PW}(\Sigma) \leq 1/(3s)$,  then the basis pursuit problem~\eqref{eq:BP:def} recovers the (true) vector of parameters, $\beta$. 
%
Let us define 
\begin{align*}
	\Gamma_{(s)} := \bigcup_{S :\; |S|\, \le\, s} \Gamma_S, \quad \text{where $\Gamma_S$ is given in~\eqref{eq:Gamma:S:def}}.
\end{align*}
An incoherence condition holds over $\Gamma_{(s)}$ iff it holds over $\Gamma_S$ for all subset $S$ of size at most $s$.

We have the following result connecting PWI and MR incoherence:
\begin{prop}\label{prop:MR:PW}
Assume  $\diag(\Sigma) = 1_p$. Then, MR incoherence~\eqref{eq:MR:incoh} holds over $\Gamma_{(s)}$ if 
\begin{align*}
\dpw(\Sigma) < \cfrac{1-\delta'}{2R(s-1)+1}.
\end{align*} 
\end{prop}

Proposition~\ref{prop:MR:PW} together with Proposition~\ref{prop:MR:RIP} stated in Section~\ref{sec:landscape}
relate two well-known conditions to the MR incoherence derived in Section~\ref{sec:MR:population}.
%
%
These two propositions roughly show that in terms of the strength, the MR incoherence is somewhere between the PW incoherence and the RIP, that is, PW incoherence implies a form of MR incoherence which in turn implies a form of RIP.